\documentclass{article}

\usepackage[latin1]{inputenc}
\usepackage{graphicx}
\usepackage{amssymb,amsfonts,amsmath}
\usepackage{html}

\newtheorem{theorem}{Theorem}

\newtheorem{corollary}[theorem]{Corollary}

\newtheorem{lemma}[theorem]{Lemma}

\newtheorem{proposition}[theorem]{Proposition}
\newtheorem{remark}[theorem]{Remark}

\newenvironment{proof}[1][Proof]{\noindent\textbf{#1.} }{\newline \hspace*{\textwidth}\hspace*{-0,4cm} \rule{0.5em}{0.5em} \vspace{0,2cm}}

\newcommand{\R}{\mathbb{R}}
\newcommand{\N}{\mathbb{N}}

\begin{document}

\title{Nonconvolution Nonlinear Integral Volterra Equations with
  Monotone Operators}
\author{M. R. Arias$^{1*}$, R. Ben\'{\i}tez$^{2*}$, V. J. Bol\'os$^1$ \\
\\
{\small $^1$ Departamento de Matem\'aticas,}\\
{\small Facultad de Ciencias, Universidad Extremadura.}\\
{\small Avda. de Elvas s/n, 06071 Badajoz, Spain.}\\
{\small e-mail\textup{: \texttt{arias@unex.es}\qquad \texttt{vjbolos@unex.es}}} \\
\\
{\small $^2$ Departamento de Matem\'aticas,}\\
{\small Centro Universitario de Plasencia, Universidad Extremadura.}\\
{\small Avda. Virgen del Puerto 2, 10600 Plasencia, Spain.}\\
{\small e-mail\textup{: \texttt{rbenitez@unex.es}}} \\
}
\footnotetext{$^*$Research partially supported by CICYT (Project BFM2001-0849), Spain.}
\date{February 2005}
\maketitle

\begin{abstract}
Some results about existence, uniqueness, and attractive behaviour of solutions for nonlinear Volterra integral equations with non-convolution kernels are presented in this paper. These results are based on similar ones about nonlinear Volterra integral equations with convolution kernels and some comparison techniques. Therefore, this paper is devoted to find a wide class of nonconvolution Volterra integral equations where their solutions behave like those of Volterra equations with convolution kernels.
\end{abstract}

\section{Introduction}

The aim of this paper is to the study of the nonlinear Volterra integral
equation
\begin{equation}
  \label{eq:vol1}
  u\left( x\right) =\int_0^xk\left( x,s\right) g\left( u\left( s\right) \right) \,ds,
\end{equation}
that will be denoted by $\left( k,g\right) $. We will assume that
the following conditions are held.
\begin{enumerate}
\item [$\mathbf{K}_1$.] The kernel $k:\R ^2\to\R ^+$ is a
  locally bounded function, such that $k\left( x,s\right) =0$ whenever $s>x$.
\item [$\mathbf{K}_2$.] For every $x\in\R $, the map $s\to
  k\left( x,s\right) $ is locally integrable, and $K\left( x\right) =\int_0^xk\left( x,s\right) \,ds$ is a
  strictly increasing function.
\item [$\mathbf{G}_1$.] The nonlinearity $g$ is a
  strictly increasing continuous function, vanishing on $(-\infty,0]$, and such
  that $g'>0$ almost everywhere.
\end{enumerate}
From now on, these conditions will be referred to as (GC).

Solutions of an equation $\left( k,g\right) $ are fixed points of the operator
$T_{kg}$, defined as
\begin{equation}
  \label{eq:op1}
  T_{kg}f\left( x\right) :=\int_0^xk\left( x,s\right) g\left( f\left( s\right) \right) \,ds.
\end{equation}
The monotone behaviour of $T_{kg}$ is an immediate consequence of $\mathbf{G}_1$ and the strictly increasing behaviour of the integral operator; i.e.,  if $f_1\leq f_2$, then
$T_{kg}f_1\leq T_{kg}f_2$. Moreover, since $g\left( 0\right) =0$, the
zero function is a solution of (\ref{eq:vol1}), known as \textit{the
trivial solution}.

The following two lemmas allow us to consider only bounded solutions on a certain interval $[0,\delta]$, for some positive $\delta$. This kind of solutions will be referred to as \textit{bounded near zero} functions.

\begin{lemma}
\label{lema:1}
Let $k$ be a kernel satisfying the following inequality,
  \begin{equation}
\label{prop:szwarc}
k\left( x,s\right) \leq k\left( y,s\right) ,\qquad \forall x\leq y,
\end{equation}
for each $s\in \R$. Then, the operator $T_{kg}$ transforms positive functions into
increasing functions.
\end{lemma}
\begin{proof}
  Let $f$ be a positive function, and let $x\leq y$. From
  $\mathbf{K}_1$, we have
\begin{equation*}
  \begin{split}
    T_{kg}f\left( x\right) =&\int_0^xk\left( x,s\right) g\left( f\left( s\right) \right) \,ds=\int_0^yk\left( x,s\right) g\left( f\left( s\right) \right) \,ds\\
\leq &\int_0^yk\left( y,s\right) g\left( f\left( s\right) \right) \,ds=T_{kg}f\left( y\right) .\\
  \end{split}
\end{equation*}
\end{proof}

\begin{lemma}
\label{lastlemma}
  Let $f$ be a positive function. Then, for every $x$ in its domain of
  definition, $T_{kg}f$ is bounded on $[0,x]$.
\end{lemma}
\begin{proof}
Let us define the auxiliary kernel,
\begin{equation*}
\overline{k}\left( x,s\right) =\max\left\{ k\left( t,s\right) :\quad 0\leq t\leq x\right\} .
\end{equation*}
The kernel $\overline{k}$ verifies the condition (\ref{prop:szwarc}) and $k\leq\overline{k}$. Then, if $f$ is a positive function, $T_{kg}f\leq T_{\overline{k}g}f$. From Lemma \ref{lema:1}, it follows that $T_{\overline{k}g}f$ is an increasing function. Thus, for every $x$ where $T_{kg}f$ is defined, $T_{kg}f$ is bounded by $T_{\overline{k}g}f\left( x\right) $ on $[0,x]$.
\end{proof}

Taking into account Lemma \ref{lastlemma}, positive solutions for
equation (\ref{eq:vol1}) are \textit{bounded near zero}. Unless otherwise stated, any function considered in this paper
will be \textit{bounded near zero}.

A particular case of equation $\left( k,g\right) $ is the well known \textit{convolution equation},
\begin{equation}
  \label{eq:vol2}
  u\left( x\right) =\int_0^x\phi\left( x-s\right) g\left( u\left( s\right) \right) \,ds.
\end{equation}
Here, the kernel is $k\left( x,s\right) =\phi\left( x-s\right) $, being $\phi$ a locally bounded
function of one real variable. This kind of kernels are known as \textit{convolution kernels}.

The existence of a nontrivial solution for convolution equations is equivalent to the existence of a nontrivial
\textit{subsolution}; i.e., a function $v$ such that $v\leq T_{\phi g}v$ \cite{Ari00,AriCas98,BusOkra90,Grip81,Myd91,Zeid90}.
Moreover, if a positive solution of (\ref{eq:vol2}) exists, then it is unique, strictly increasing, continuous and a global attractor
of any positive and measurable function $f$ (see, for instance, \cite{Ari00,AriBen01,AriCas99,Szw92}. Recall that a solution is a global attractor of a positive measurable function $f$ if the sequence $(T^n_{\phi g}f)_{n\in \N }$ converges to that solution, where $T_{\phi g}^n$ denotes the composition of $T_{\phi g}$ with itself $n$ times.

Szwarc, in \cite{Szw92}, presented several results about existence, uniqueness, and attracting behaviour of solutions for nonconvolution Volterra integral equations. In that paper, the author uses different techniques and ideas which appear in many results concerning the existence, uniqueness, and attracting behaviour of solutions for convolution equations. Our aim in this paper is the same. That is, to study how the results known for the
convolution equation (\ref{eq:vol2}), can be used in order to
obtain properties for the solutions of the nonconvolution equation
(\ref{eq:vol1}). The hypotheses considered in this paper are weaker than those considered by Szwarc in \cite{Szw92}.

\section{Existence of solutions}
\label{sec2}

As mentioned above, for nonlinear Volterra integral equations of convolution type there is a strong relation between the existence of subsolutions and the existence of nontrivial solutions. First we will show that, also for
nonconvolution equations, the existence of solutions and the existence of subsolutions are equivalent.

Throughout this section, we will assume that equation $\left( k,g\right) $ verifies conditions (GC).

\begin{theorem}\label{teo:subsol}
  There is a solution for the equation (\ref{eq:vol1}) if and only if equation (\ref{eq:vol1}) admits a subsolution.
\end{theorem}
\begin{proof}
The sufficient condition is immediate, because every solution of equation (\ref{eq:vol1}) is a subsolution.

To prove the necessary condition, let us consider a positive subsolution of (\ref{eq:vol1}), $v$. First, we want to note that, by Lemma \ref{lastlemma}, subsolutions of (\ref{eq:vol1}) are necessarily \textit{bounded near zero}. So, there
exist positive $\delta_1$ and $M$, such that
\begin{equation}
\label{eq:des1}
v\leq M,\qquad \text{ on } [0,\delta_1].
\end{equation}
Now, we need to prove that $M$ is a \textit{supersolution near zero}, which is equivalent to prove the existence of a positive $\delta _2$, such that
\begin{equation}
\label{eq:des2}
T_{kg}M\leq M,\qquad \text{on }[0,\delta_2].
\end{equation}
Taking into account conditions $\mathbf{K}_1$ and $\mathbf{K}_2$, we have that $K\left( 0\right) =0$ and $\lim_{x\to 0+}K\left( x\right) =0$. Therefore, since $T_{kg}M\left( x\right) =g\left( M\right) K\left( x\right) $, the existence of $\delta_2$ is guaranteed.

Let us define $\delta=\min\left\{ \delta_1,\delta_2\right\} $. From
  (\ref{eq:des1}) and (\ref{eq:des2}), we have
  \begin{equation*}
    v\leq T_{kg}v\leq T_{kg}M\leq M,\qquad \text{on }[0,\delta].
  \end{equation*}
  Note that $(T^n_{kg}v)_{n\in \N }$ is a nondecreasing sequence
  bounded from above by $M$. Thus, we can define the pointwise limit
  \begin{equation*}
    u\left( x\right) :=\lim_{n\to \infty }T^n_{kg}v\left( x\right) ,\qquad \forall x\in [0,\delta].
  \end{equation*}
  For each $x\in [0,\delta]$, we consider the sequence $\left( \phi_n\right) _{n\in \N }$, where
  \begin{equation*}
    \phi_n\left( s\right) =k\left( x,s\right) g\left( T^n_{kg}v\left( s\right) \right) .
  \end{equation*}
  By the monotone convergence theorem, the function
\begin{equation*}
u\left( x\right) = \lim_{n\to \infty }\int_{\R }\phi _n ^x \left( s\right) \,ds
\end{equation*}
exists on $[0,\delta]$ and is a solution of equation (\ref{eq:vol1}).
\end{proof}

Note that the necessary condition of last lemma remains true when you assume just the existence of a subsolution \textit{near zero}, i.e., the existence of a function $v$ and a positive $\delta _0$ such that
\begin{equation*}
v\leq T_{kg}v,\qquad \text{ on }[0,\delta _0].
\end{equation*}
In this case, it only would be necessary to change, in the proof of the necessary condition, the definition of $\delta $; the new definition would be $\delta =\min \left\{ \delta _0,\delta _1,\delta _2\right\} $.

As we mentioned in the introduction, Volterra integral equations of convolution kind are a particular case of equation (\ref{eq:vol1}). There are many results about the existence and uniqueness of solutions for convolution Volterra integral equations \cite{Ari00,BusOkra90,Myd91,Zeid90,AriBen03a,Myd99}. Some of the foremost techniques to study Volterra integral equations are comparison techniques \cite{Zeid90,Ask91,Kar00}. The rest of this section is devoted to the use of such techniques in order to establish a relation between existence results for \textit{convolution equations} and for equation (\ref{eq:vol1}). To do it, we will need to show that any locally bounded kernel can be bounded from above and below by \textit{convolution kernels}, on every bounded region of $\R ^2$.

Since our interest is to relate equation (\ref{eq:vol1}) with \textit{Volterra integral equations of convolution kind}, at a first stage, it would be natural to consider kernels $k:\R^2\rightarrow \R^+$ verifying $k\left( x,s\right) =k\left( x+\lambda,s+\lambda\right) $ for all $\lambda \in \R $ and $\left( x,s\right) \in \R^2$. Such kernels will be referred to as \textit{invariant kernels}. Note that convolution kernels are invariant because there is a function $\phi:\R \rightarrow \R^+$ such that $k\left( x,s\right) =\phi\left( x-s\right) $. Next, we are going to see that any \textit{invariant kernel} is a \textit{convolution kernel}. Let $k$ be an invariant kernel, then
$k\left( x,s\right) =k\left( x-s,0\right) $, for all $\left( x,s\right) \in \R^2$; so defining $\phi\left( x\right) =k\left( x,0\right) $, we have
$k\left( x,s\right) =\phi\left( x-s\right) $. Thus, both families, invariant and convolution kernels are the same.

Now, let us consider a kernel $k$ satisfying $\mathbf{K}_1$, and let us
study equation (\ref{eq:vol1}) in an interval $[0,x_0]$, for a given $x_0>0$. First, we define a couple of auxiliary functions,
\begin{equation}\label{eq:func_min}
  \phi _{x_0}\left( x\right) =\min \left\{ k\left( \left( 1-\lambda \right) x+\lambda x_0,\lambda \left( x_0-x\right) \right) :
\lambda \in [0,1]\right\}
\end{equation}

\noindent and

\begin{equation}\label{eq:func_max}
    \psi_{x_0}\left( x\right) =\max \left\{ k\left( \left( 1-\lambda \right) x+\lambda x_0,\lambda\left( x_0-x\right) \right) :
\lambda \in[0,1]\right\}.
\end{equation}

Let $\mathcal{T}_{x_0}$ be the right triangle determined by $\left( 0,0\right) $,
$\left( x_0,0\right) $ and $\left( x_0,x_0\right) $. For every $x\in [0,x_0]$, $\phi\left( x\right) $ and
$\psi\left( x\right) $ are the minimum and the maximum, respectively, of $k$ on the
segment $l_x$, determined by the intersection of $\mathcal{T}_{x_0}$ and the graph of $y\left( s\right) =s-x$. So, we have
\begin{equation}
\label{eq9}
\phi_{x_0}\left( x_1-s_1\right) \leq k\left( x_1,s_1\right) \leq\psi_{x_0}\left( x_1-s_1\right) ,
\end{equation}
for any $\left( x_1,s_1\right) \in \mathcal{T}_{x_0}$, because $\left( x_1,s_1\right) $ is on $l_{x_1-s_1}$.

From (\ref{eq9}) and Theorem \ref{teo:subsol}, it
follows that the existence of a solution for a equation
$\left( \phi_{x_0},g\right) $ implies the existence of solutions for equation
$\left( k,g\right) $ and
$\left( \psi_{x_0},g\right) $. In general, the converse is not true. But if we assume the existence of a positive constant $c$ such that $\psi_{x_0}\leq c\phi_{x_0}$, the following inequalities hold,
\begin{equation*}
  \phi_{x_0}\left( x-s\right) \leq k\left( x,s\right) \leq c\phi_{x_0}\left( x-s\right) ;
\end{equation*}
and, therefore, by Theorem \ref{teo:subsol}, the existence of solutions for $\left( k,g\right) $ is
equivalent to the existence of solutions for $(\phi_{x_0},g)$. There are different cases in which such constant can be found. For instance, when
\begin{equation}
\label{cond10}
  \lim_{x\to 0^+}\frac{\psi_{x_0}\left( x\right) }{\phi_{x_0}\left( x\right) }=l\in[0,+\infty).
\end{equation}

What we have proved in the last part of this section is the following result.
\begin{theorem}
Let $\left( k,g\right) $ be a nonconvolution equation satisfying (GC), and let $\phi_{x_0}$ and $\psi_{x_0}$ be defined as in (\ref{eq:func_min}) and (\ref{eq:func_max}). Then, the existence of a solution for equation $\left( \phi_{x_0},g\right) $ implies the existence of a solution for equation $\left( k,g\right) $. Moreover, if condition(\ref{cond10}) holds, the equation $\left( k,g\right) $ has a solution if and only if either equation $\left( \phi_{x_0},g\right) $ or $\left( \psi_{x_0},g\right) $ have a solution.
\end{theorem}

Let us see a couple of examples about how to use the techniques described in this section to prove the existence of solutions for $\left( k,g\right) $.

\noindent\subparagraph{Example 1.}
Let the equation $\left( k,g\right) $ be
\begin{equation}\label{eq:ex1}
u\left( x\right) =\int_0^x\left( a^{x+s}+1\right) \sqrt{2u\left( s\right) }\,ds,\qquad a>0,\quad x\geq 0.
\end{equation}

Let us consider an arbitrary positive constant $x_0>0$, and restrict the problem to the interval $[0,x_0]$.
Consider the triangle
\begin{equation*}
\mathcal{T}_{x_0}=\left\{ \left( x,s\right) \in\R^2:0\leq x\leq x_0,\,0\leq s\leq x\right\} .
\end{equation*}

Since the kernel is increasing with respect both variables, the functions $\phi$ and $\psi$, defined in (\ref{eq:func_min}) and (\ref{eq:func_max}), are $\phi\left( x\right) =a^x+1$ and $\psi\left( x\right) =a^{2x_0-x}+1$.

We also have
\begin{equation*}
\lim_{x\to 0^+}\frac{\psi\left( x\right) }{\phi\left( x\right) }=\frac{a^{2x_0}+1}{2}\in [0,+\infty );
\end{equation*}
thus, condition (\ref{cond10}) holds, and therefore, the existence of solutions for equation (\ref{eq:ex1}) is equivalent to the existence of a solution for the equation
\begin{equation}\label{eq:convex1}
u\left( x\right) =\int _0^x \left( a^{x-s}+1\right) \sqrt{2u\left( s\right) }\,ds,\qquad x\in[0,x_0].
\end{equation}
It can be easily checked that (\ref{eq:convex1}) verifies some conditions for the existence of solutions for convolution equations given in \cite{Ari00}. Hence, the nonconvolution equation $\left( k,g\right) $ has a solution.

\noindent\subparagraph{Example 2.} Let the equation $\left( k,g\right) $ be
\begin{equation*}
u\left( x\right) =\int _0^x x\left( x-s\right) u\left( s\right) ^{\beta }\, ds,\qquad x\in [0,L],\quad \beta \in \left( 0,1\right) .
\end{equation*}

In this case, the kernel is $k\left( x,y\right) = x\left( x-y\right) $. As in the last example, it is possible to find the expressions of the functions $\phi $ and $\psi $ when we restrict the problem to the region $\mathcal{T}_{x_0}$. Here, we have $\phi \left( x\right) =x^2$ and $\psi \left( x\right) =x_0x$. For such functions, we find that
\begin{equation*}
\lim _{x\to 0^+}\frac{\psi\left( x\right) }{\phi \left( x\right) }=\lim_{x\to 0^+}\frac{x_0}{x}=+\infty .
\end{equation*}
Hence, condition (\ref{eq:convex1}) does not hold. Nevertheless, it is immediate to check that both equations, $\left( x^2,x^\beta \right) $ and $\left( x_0x,x^\beta \right) $ have a solution. Indeed, it is possible to obtain the solutions in closed form. The functions
\begin{equation*}
u_{\phi }\left( x\right) =B\left( 3,\frac{3\beta +1}{1-\beta }\right) ^{1/\left( 1-\beta \right) } x^{3/\left( 1-\beta \right) }
\end{equation*}
and
\begin{equation*}
u_{\psi }\left( x\right) =\left( x_0B\left( 2,\frac{2\beta +1}{1-\beta }\right) \right)^{1/\left( 1-\beta \right) } x^{2/\left( 1-\beta \right) }
\end{equation*}
are the solutions for equations $\left( x^2,x^\beta \right) $ and $\left( x_0x,x^\beta \right) $ respectively. Therefore, every solution for equation $\left( k,g\right) $ lies between $u_{\phi }$ and $u_{\psi }$.

\section{Uniqueness}
For convolution equations with locally bounded kernels, under very weak assumptions, nontrivial solutions are unique, see \cite{AriBen01}. Our aim in this section is to prove the uniqueness of nontrivial solutions for nonconvolution equations. To do it, we will consider the following additional hypotheses on the kernel.

\begin{enumerate}
\item[$\mathbf{K}_3$.] The function $K\left( x\right) =\int _0^x k\left( x,s\right) \, ds$ is continuous.
\item[$\mathbf{K}_4$.] For every $\left( x,s\right) \in \R ^2$ and $\lambda \geq 0$, $k\left( x,s\right) \leq k\left( x+\lambda ,s+\lambda \right) $.
\end{enumerate}

\begin{lemma}\label{lemma:2}
Let us suppose that, in addition to (GC), equation $\left( k,g\right) $ also verifies $\mathbf{K}_3$. Then, the operator $T_{kg}$ transforms
  bounded functions into continuous functions.
\end{lemma}
\begin{proof}
Let $f$ be a positive function bounded from above by $M$. Let
 $x_1\leq x_2$, then, since $k\left( x,s\right) =0$ whenever $s>x$, we have
\begin{eqnarray*}
T_{kg}f\left( x_2\right) -T_{kg}f\left( x_1\right) & = & \int_0^{x_2}k\left( x_2,s\right) g\left( f\left( s\right) \right) \,ds-\int_0^{x_1}
k\left( x_1,s\right) g\left( f\left( s\right) \right) \,ds \\
& = & \int_0^{x_2}\left( k\left( x_2,s\right) -k\left( x_1,s\right) \right) g\left( f\left( s\right) \right) \,ds \\
& \leq & g\left( M\right) \int_0^{x_2}k\left( x_2,s\right) -k\left( x_1,s\right) \,ds \\
& = & g\left( M\right) \left( K\left( x_2\right) -K\left( x_1\right) \right) . \\
\end{eqnarray*}
The continuity of $T_{kg}f$ is immediate from the continuity of $K$.
\end{proof}

The next corollary is followed from Lemma \ref{lastlemma} and the last result.

\begin{corollary}\label{cor:1}
Every solution of equation $\left( k,g\right) $ is a continuous function.
\end{corollary}

The proof of the next lemma has been adapted from a paper due to Mydlardzyc \cite{Myd99}, where a similar result was proved for Abel integral equations. Here, we have used the ideas presented in \cite{Myd99}, and extended them to nonconvolution equations.

\begin{lemma}\label{lemma:3}
Let us suppose that, in addition to (GC), equation $\left( k,g\right) $ also verifies $\mathbf{K}_3$ and $\mathbf{K}_4$. Then, every continuous subsolution of equation $\left( k,g\right) $ is bounded from above by any solution of equation $\left( k,g\right) $.
\end{lemma}
\begin{proof}
Let $v$ and $u$ be a subsolution and a solution of equation $\left( k,g\right) $,
 respectively. First, we will show that, for every $c>0$, the function
\begin{equation*}
v_c\left( x\right) =\begin{cases}
0,&\text{if }x\in [0,c]\\
v\left( x-c\right) ,&\text{if } x>c, \\
       \end{cases}
\end{equation*}
is also a subsolution of equation $\left( k,g\right) $.
For $x\in [0,c]$, this is trivial since $v_c\left( x\right) =T_{kg}v_c\left( x\right) =0$. For $x>c$,
\begin{equation}
\label{eq14}
v_c\left( x\right) =v\left( x-c\right) \leq T_{kg}v\left( x-c\right) =\int_0^{x-c}k\left( x-c,s\right) g\left( v\left( s\right) \right) \,ds.
\end{equation}
Since $k$ verifies $\mathbf{K}_4$, making the change of variable $t=s+c$ in the last integral, (\ref{eq14}) takes the form
\begin{eqnarray*}
  v_c\left( x\right) & \leq & \int _c^x k\left( x-c,t-c\right) g\left( v\left( t-c\right) \right) \,dt \\
& \leq & \int _c^x k\left( x,t\right) g\left( v_c\left( t\right) \right) \, dt=\int _0^x k\left( x,t\right) g\left( v_c\left( t\right) \right) \, dt = T_{kg}v_c\left( x\right) .\\
\end{eqnarray*}
Now, let us compare $v_c$ and $u$. For $0<x<c$, it is obvious that
$0=v_c\left( x\right) <u\left( x\right) $. Since $v_c$ and $u$ are continuous, there exists an interval
$[0,x_0)$, with $x_0>c$, where $v_c\leq u$. Then,
\begin{eqnarray*}
u\left( x_0\right) -v_c\left( x_0\right) & \geq & \int _0^{x_0} k\left( x_0,s\right) \left[ g\left( u\left( s\right) \right) -g\left( v_c\left( s\right) \right) \right] \, ds \\
& > & \int _0^c k\left( x_0,s\right) \left[ g\left( u\left( s\right) \right) -g\left( v_c\left( s\right) \right) \right] \, ds \\
& = & \int _0^c k\left( x_0,s\right) g\left( u\left( s\right) \right) \, ds>0.
\end{eqnarray*}
Analogously, it can be assured that $v_c<u$ in the whole domain of the solution.

Finally, since $v_c<u$ for every positive $c$, taking limits as $c\to 0^+$, we obtain that $v\leq u$.
\end{proof}

A consequence of this Lemma is the uniqueness of positive solutions for the
 equation $\left( k,g\right) $.
 
\begin{theorem}
Under the hypotheses of Lemma \ref{lemma:3}, equation $\left( k,g\right) $ has at most
 one positive solution.
\end{theorem}
\begin{proof}
Since every solution can be considered as a special
case of continuous subsolution, this proof is trivial.

Let $u_1$ and $u_2$ be two solutions of $\left( k,g\right) $.
Considering $u_1$ as a continuous subsolution, by Lemma \ref{lemma:3} $u_1\leq u_2$;
and considering $u_2$ as a continuous subsolution, $u_2\leq u_1$, so
$u_1\equiv u_2$.
\end{proof}

\section{Attracting behavior}

In this section, we are going to study
the attracting behaviour of the solutions for
the equations $\left( k,g\right) $ verifying conditions (GC).

Recall that in Section \ref{sec2}, for a kernel $k$ satisfying $\mathbf{K}_1$, there were defined the functions
\begin{equation*}
  \phi _{x_0}\left( x\right) =\min \left\{ k\left( \left( 1-\lambda \right) x+\lambda x_0,\lambda \left( x_0-x\right) \right) :
\lambda \in [0,1]\right\}
\end{equation*}

\noindent and

\begin{equation*}
    \psi_{x_0}\left( x\right) =\max \left\{ k\left( \left( 1-\lambda \right) x+\lambda x_0,\lambda\left( x_0-x\right) \right) :
\lambda \in[0,1]\right\}.
\end{equation*}
In that section, nonconvolution equations $\left( k,g\right) $ were studied in some arbitrary interval $[0,x_0]$ using the auxiliary convolution equations $\left( \phi _{x_0},g\right) $ and $\left( \psi _{x_0},g\right) $.

In order to simplify the notation, unless otherwise stated, $\phi _{x_0}$ and $\psi _{x_0}$ will be referred to as $\phi $ and $\psi $, respectively.

The proofs of the results presented in this section are mainly based on the attracting character of the solutions for the equations $\left( \phi ,g\right) $ and $\left( \psi ,g\right) $, and some standard comparison techniques. Throughout this section, we will assume the existence of solutions for equations $\left( \phi ,g\right) $ and $\left( \psi ,g\right) $, that will be denoted by $u_{\phi }$ and $u_{\psi }$ respectively.

Note that $u_{\phi }$ and $u_{\psi }$ are unique and global attractors of all positive and measurable functions (see \cite{AriBen01}). Moreover, as we saw in Section \ref{sec2}, from (\ref{eq9}) and Theorem \ref{teo:subsol}, the existence of a solution for an equation $\left( \phi ,g\right) $ implies the existence of solutions for $\left( k,g\right) $. These solutions are comparable functions, as we will see in the next result.

\begin{lemma}\label{prop:u<upsi}
Let $u$ be a solution of the nonconvolution equation $\left( k,g\right) $. Then  $u_{\phi }\leq u\leq u_{\psi }$.
\end{lemma}
\begin{proof}
Note that
\begin{equation}
\label{eq15}
T_{\phi g}\leq T_{kg}\leq T_{\psi g},
\end{equation}
because $\phi \leq k\leq \psi $. Thus, $T_{\phi g}u\leq u=T_{kg}u\leq T_{\psi g}u$, and then, for every natural $n$,
\begin{equation*}
T^n_{\phi g}u\leq u\leq T^n_{\psi g}u.
\end{equation*}

Since both, $u_{\phi }$ and $u_{\psi }$, are global attractors, the sequences
  $(T^n_{\phi g}u)_{n\in \N }$ and $(T^n_{\psi g}u)_{n\in \N }$ converge to $u_{\phi }$ and $u_{\psi }$,
  respectively. Thus, taking limits as $n$ tends to $\infty $, we have $u_{\phi }\leq u\leq u_{\psi }$.
\end{proof}

\begin{proposition}
\label{prop:upsi}
  The sequence $(T^n_{kg}u_{\psi })_{n\in \N }$ converges to the maximum solution of the equation $\left( k,g\right) $.
\end{proposition}
\begin{proof}
  By Lemma \ref{prop:u<upsi}, $u\leq u_\psi$. Thus, from the monotony of the operators $T_{k g}$
  and $T_{\psi g}$, it follows that
  \begin{equation*}
    u=T_{kg}u\leq T_{kg}u_{\psi }\leq T_{\psi g}u_{\psi }=u_{\psi }.
  \end{equation*}
  Hence, for every $x\geq 0$, the decreasing sequence
  $(T^n_{kg}u_{\psi }\left( x\right) )_{n\in \N }$ is bounded from below by $u\left( x\right) $, so it converges pointwisely to a function
  \begin{equation*}
    u_{\max }\left( x\right) :=\lim_{n\to \infty }T^n_{kg}u_{\psi }\left( x\right) =\inf \left\{ T^n_{kg}u_{\psi }\left( x\right) :n\in\N\right\} .
  \end{equation*}
  By the monotone convergence theorem, we can assure that $u_{\max }$ is a solution of the equation $\left( k,g\right) $;
  moreover from the way of constructing $u_{\max }$, it is immediate that it is the maximum solution.
\end{proof}

With a similar proof we obtain an analogous result for the minimum solution.
\begin{proposition}
\label{prop:uphi}
  The sequence $(T^n_{kg}u_{\phi })_{n\in \N }$ converges to the minimum solution of the equation $\left( k,g\right) $.
\end{proposition}

Now we are in position to give a result on the attracting character of the maximum and minimum solutions of equation $\left( k,g\right) $.

\begin{theorem}
The maximum (resp. minimum) solution of the equation $\left( k,g\right) $ attracts globally any measurable function bounded fom below (resp. above) by the maximum (resp. minimum) solution.
\end{theorem}
\begin{proof}
We shall prove the theorem for the maximum solution. For the minimum solution, there can be used analogous arguments.

Let $u_{\max }$ denote the maximum solution of the equation $\left( k,g\right) $, and let $f$ be a measurable function such that $u_{\max }\leq f$. We have to show that $(T^n_{kg}f)_{n\in \N }$ converges to $u_{\max }$.

From (\ref{eq15}) and the increasing character of the operators $T_{kg}$ and $T_{\psi g}$, we obtain
\begin{equation*}
u_{\max }=T^n_{kg}u_{\max }\leq T^n_{kg}f\leq T^n_{\psi g}f,\qquad \forall n\in \N .
\end{equation*}
Thus, for every $x\geq 0$, the sequence $(T^n_{kg}f\left( x\right) )_{n\in \N }$ is bounded from below by $u_{\max }\left( x\right) $, and from above by the  sequence $(T^n_{\psi g}f\left( x\right) )_{n\in \N }$, which, as said above, converges to $u_\psi \left( x\right) $. Then, the set of accumulation points of the sequence $(T^n_{kg}f\left( x\right) )_{n\in \N }$, denoted by $\Omega _f\left( x\right) $, verifies $u_{\max }\left( x\right) \leq \Omega_f\left( x\right) \leq u_\psi \left( x\right) $.

To finish the proof it suffices to show that $\Omega_f\left( x\right) =\left\{ u_{\max }\left( x\right) \right\} $. This is obvious, because $\Omega_f\left( x\right) $ is invariant under $T_{kg}$ and, by Proposition \ref{prop:upsi}, the sequence $(T^n_{kg}u_\psi )_{n\in \N }$ converges to $u_{\max }$. Hence, $u_{\max }\left( x\right) \leq \Omega_f\left( x\right) \leq u_{\max}\left( x\right) $.
\end{proof}

\begin{remark}
Note that if we could assure the uniqueness of solutions for the nonconvolution equation $\left( k,g\right) $, then the maximum and the minimum solutions are the same. In that case, a simple comparison reasoning guarantees that the unique solution is a global attractor of any  positive and measurable function.
\end{remark}

\section{Final Remarks}
For convolution equations, there are a lot of results about existence and uniqueness of continuous solutions with no other assumptions on the kernel than the local integrability. Just mention, for example, the theory of Abel integral equations.

When nonconvolution equations are considered, a wide range of situations appears. The aim of the following examples is to illustrate such variety. In the first example, it is shown that if $\mathbf{K}_3$ does not hold, then the only continuous solution for equation $\left( k,g\right) $ is the trivial one. In the second example, we will see that if the kernel does not verify $\mathbf{K}_1$, then we cannot guarantee the uniqueness of solutions.

\subsection{An Equation with Discontinuous Solutions}

Let us consider equation $\left( k,g\right) $ with $g$ verifying $\mathbf{G}_1$ and $k$ a strictly increasing function of the variable $x$, that is, for any fixed $s$, the function $x\mapsto k\left( x,s\right) $ is strictly increasing. We also assume that $k\left( x,s\right) $ has a simple discontinuity at $x_0$ in the following sense. Let us define
\begin{equation*}
k_0^-\left( s\right) :=\lim _{x\to x_0 \,\!^-}k\left( x,s\right) \qquad \text{and}\qquad k_0^+\left( s\right) :=\lim _{x\to x_0 \,\!^+}k\left( x,s\right) ;
\end{equation*}
then, for every $s$, $k_0^-\left( s\right) <k_0^+\left( s\right) $. Let $u$ be a solution for equation $\left( k,g\right) $. If $x<x_0$, then
\begin{equation*}
u\left( x\right) =\int _0^x k\left( x,s\right) g\left( u\left( s\right) \right) \, ds<\int _0^{x_0}k_0^-\left( s\right) g\left( u\left( s\right) \right) \, ds.
\end{equation*}
In a similar way, for $x>x_0$,
\begin{equation*}
u\left( x\right) =\int _0^x k\left( x,s\right) g\left( u\left( s\right) \right) \, ds>\int _0^{x_0}k_0^+\left( s\right) g\left( u\left( s\right) \right) \, ds.
\end{equation*}
Hence, taking lateral limits, we obtain
\begin{equation*}
\lim _{x\to x_0 \,\!^-}u\left( x\right) \leq \int _0^{x_0}k_0^-\left( s\right) g\left( u\left( s\right) \right) \, ds < \int _0^{x_0}k_0^+\left( s\right) g\left( u\left( s\right) \right) \, ds \leq \lim _{x\to x_0 \,\!^+} u\left( x\right) .
\end{equation*}
Thus, $u$ also has a simple discontinuity at $x_0$. Note that with an analogous proof, we can show that the function $K\left( x\right) =\int _0^x k\left( x,s\right) \, ds$ has a simple discontinuity at $x_0$. So, condition $\mathbf{K}_3$ does not hold.

\subsection{An Equation with Multiple Solutions}

Let us consider the equation
\begin{equation}
\label{eq16}
u\left( x\right) = \int _0^x \left( x-s\right) ^{\alpha }s^{-\alpha -1}g\left( u\left( s\right) \right) \, ds,\qquad \alpha \in (-1,0).
\end{equation}
It is an equation of type (\ref{eq:vol1}), where $k\left( x,s\right) = \left( x-s\right) ^{\alpha }s^{-\alpha -1}$. We are considering a nonlocally bounded kernel; hence, condition $\mathbf{K}_1$ does not hold. As $-\alpha -1<0$,
\begin{equation*}
k\left( x+\lambda ,s+\lambda \right) =\left( x-s\right) ^{\alpha }\left( s+\lambda \right) ^{-\alpha -1}<\left( x-s\right) ^{\alpha }s^{-\alpha -1}=k\left( x,s\right) ,
\end{equation*}
for every $\lambda >0$. Thus, condition $\mathbf{K}_4$ is not verified either.

Now, let us suppose that equation (\ref{eq16}) has a positive constant solution, $u\left( x\right) =M$, for some $M$. Then,
\begin{eqnarray*}
M & = & \int _0^x \left( x-s\right) ^{\alpha }s^{-\alpha -1}g\left( M\right) \, ds=g\left( M\right) \int _0^x \left( x-s\right) ^{\alpha }s^{-\alpha -1}\, ds \\
& = & g\left( M\right) B\left( \alpha +1,-\alpha \right) .
\end{eqnarray*} 
Therefore, $u$ is a solution for equation (\ref{eq16}) if and only if $M$ is a root of the scalar equation
\begin{equation}
\label{eq17}
M-g\left( M\right) B\left( \alpha +1,-\alpha \right) =0.
\end{equation}
Equation (\ref{eq17}) depends on $g$. Then, the number of its roots also depends on $g$. It is not difficult to find nonlinearities in order to obtain any fixed number of roots for (\ref{eq17}). For instance, let
\begin{equation*}
g\left( x\right) =\frac{1}{B\left( \alpha +1,-\alpha \right) }\left( x^3 -3x^2 +3x\right) ,
\end{equation*}
(note that $g$ verifies $\mathbf{G}_1$). In this case, equation (\ref{eq17}) becomes
\begin{equation*}
x^3 -3x^2 +2x=0.
\end{equation*}
The roots of this equation are: $0$, $1$ and $2$; and the positive constant functions $u_1\left( x\right) =1$ and $u_2\left( x\right) =2$ are two solutions for equation (\ref{eq16}).

\end{document}